\documentclass[11pt]{article}

\RequirePackage[OT1]{fontenc}
\RequirePackage{amsthm,amsmath,graphicx,epsfig,amssymb,comment}
\RequirePackage[numbers]{natbib}
\RequirePackage[colorlinks,citecolor=blue,urlcolor=blue]{hyperref}

\title{On robust stopping times for detecting changes in distribution}

\author{
Y. Golubev\thanks{Aix-Marseille Universit\'e, Centre de Math\'ematiques et Informatique, 
39, rue F. Joliot Curie, 13453 Marseille, France and Institute for Information Transmission Problems, Moscow, Russia.  \texttt{e-mail: golubev.yuri@gmail.com}}
\ and M. Safarian
\thanks{ Karlsruhe Institute of Technology,
Department of Economics,
Bl\"ucherstrasse 17,
76185 Karlsruhe,
Germany. \texttt{e-mail: mher.safarian@kit.edu}}
}

\date{}

\newtheorem{thm}{Theorem}[section]
\newtheorem{lemma}{Lemma}[section]

\begin{document}

\maketitle



\begin{abstract}
 Let $X_1,X_2,\ldots $ be independent random variables  observed sequentially and such that $X_1,\ldots,X_{\theta-1}$ have a common   probability density $p_0$, while $X_\theta,X_{\theta+1},\ldots $ are all distributed  
according to  $p_1\neq p_0$. It is assumed that   $p_0$ and $p_1$ are known, but 
the time change $\theta\in \mathbb{Z}^+$ is unknown and the goal is   
to construct a stopping time $\tau$ that detects the change-point $\theta$ 
as soon as possible.  The existing approaches to this problem  rely essentially on some 
a priori information about $\theta$. For instance, in Bayes approaches, it is assumed that $\theta$ is a random variable with a known probability distribution.   In methods related to hypothesis 
testing, this a priori information is hidden in  the so-called average run length.  The main goal in this paper  is to construct  stopping times which do not make use of a priori information about $\theta$, but have nearly Bayesian detection delays.  More precisely, we propose stopping times solving approximately the following  problem: 
\begin{equation*}
\begin{split}
&\quad \Delta(\theta;\tau^\alpha)\rightarrow\min_{\tau^\alpha}\quad  
\textbf{subject to}\quad  \alpha(\theta;\tau^\alpha)\le \alpha
\   \textbf{ for any}\ \theta\ge1,
\end{split}
\end{equation*}
where  $\alpha(\theta;\tau)=\mathbf{P}_\theta\bigl\{\tau<\theta \bigr\}$ is  \textit{the false alarm probability}  and $\Delta(\theta;\tau)=\mathbf{E}_\theta(\tau-\theta)_+$ is  \textit{the average detection delay},
and explain  why such stopping times are robust w.r.t. a priori information about $\theta$.


\end{abstract}

\section{Introduction}
Let $X_1,X_2,\ldots $ be independent random variables  observed sequentially. It is assumed $X_1,\ldots,X_{\theta-1}$ have a common probability density $p_0(x),\, x\in \mathbb{R}^d$, while $X_\theta,X_{\theta+1},\ldots $ are all distributed  
according to a probability density $p_1(x)\, x\in \mathbb{R}^d$. This paper deals with the simplest change-point detection problem  where it is supposed  $p_0$ and $p_1$ are known, but 
the time change $\theta\in \mathbb{Z}^+$ is unknown, and the goal is to construct a stopping time $\tau\in \mathbb{Z}^+$ that detects   $\theta$ 
as soon as possible. 
The existing approaches to this problem  rely essentially on some 
a priori information about $\theta$. For instance, in Bayes approaches, it is assumed that $\theta$ is a random variable with a known probability distribution, see e.g.  \cite{Shir}.   In methods related to hypothesis 
testing, this a priori information is hidden in  the so-called average run length, see  e.g. \cite{Page}. Our main goal in this paper  is to construct robust stopping times which do not make use of a priori information about $\theta$, but have detection delays close to  Bayes ones. 

In order to be more precise, denote by 
$\mathbf{P}_\theta$ the probability distribution of $(X_1,\ldots,X_{\theta-1},
X_{\theta},\ldots)^\top$ and by $\mathbf{E}_\theta$ the expectation with respect to this measure. In this paper, we characterize
 $\tau$ with the help of two   functions in $\theta$:
\begin{itemize}
\item 
\textit{false alarm probability} 
$$\alpha(\theta;\tau)=\mathbf{P}_\theta \bigl\{\tau<\theta\bigr\};$$ 
\item \textit{average detection delay} 
\[
\Delta(\theta;\tau)= \mathbf{E}_\theta(\tau-\theta)_+,
\quad  \text{where}\ (x)_+=\max\{0,x\},
\]
\end{itemize}
and 
our goal is to construct stopping times  solving   the following  problem:
\begin{equation}\label{eq1}
\begin{split}
&\Delta(\theta;\tau^\alpha)\rightarrow\min_{\tau^\alpha} \quad
\textbf{subject to}\quad 
\alpha(\theta;\tau^\alpha)\le \alpha\quad 
 \textbf{for any} \quad \theta \ge1.
\end{split}
\end{equation}

The main difficulty in this problem is related to the fact that for a given stopping time $\tau^\alpha$  the average delay  $\Delta(\theta;\tau^\alpha)$ depends on $\theta$. This means that in order to compare two stopping times $\tau_1^\alpha$ and $\tau_2^\alpha$, one has to compare two  functions in $\theta\in \mathbb{Z}^+$. Obviously, this is not feasible from a   mathematical viewpoint and the principal objective in this paper  is to propose    stopping times providing good approximative solutions to  \eqref{eq1}. 
Notice also here that similar problems are common and well-known in statistics and there are reasonable  approaches to obtain their  solutions.  

In change-point detection, there are two standard methods   for constructing stopping times. 
\begin{itemize} 
\item
 \textbf{A Bayes approach.} 
The first Bayes change detection problem was stated in \cite{GR} 
for on-line quality control problem for continuous technological processes.
In detecting changes in distributions this approach assumes that $ \theta $ is a random variable  with a known  distribution
 $$
\pi_m=\mathbf{P}\{\theta=m\},\quad m=1,2,\ldots,  
$$ 
and the goal is to construct  a stopping time $\tau_\pi^\alpha$ that solves the averaged version of \eqref{eq1}, i.e., 
\begin{equation}\label{eq2}
\begin{split}
&\sum_{m=1}^\infty\pi_m\Delta(m;\tau_\pi^\alpha)\rightarrow \min_{\tau_\pi^\alpha}\quad  \textbf{subject to}\quad 
\sum_{m=1}^\infty\pi_m \alpha(m;\tau_\pi^\alpha)\le \alpha.
\end{split}
\end{equation}
Emphasize  that in contrast to \eqref{eq1}, this problem is well defined from a mathematical viewpoint, but its solution depends on  a priori law $\pi$.

\item  \textbf{A hypothesis testing approach.} 
The first  non-Bayesian change detection algorithm based on sequential hypothesis testing  was proposed in \cite{Page}.
Denote by $X^n=(X_1,\ldots,X_n)^\top$ the observations till moment $n$. The main idea in this approach is to test sequentially 
\begin{equation}\label{eq3}
\begin{split}
&\textsl{simple hypothesis }\\
&{\rm H}_0^n : X^n \sim \prod_{i=1}^n{p}_0( x_i)\\
&\textsl{vs.  compound alternative}\\
&{\rm H}_1^n : X^n \sim \prod_{i=1}^{m-1}{p}_0( x_i)\prod_{i=m}^n {p}_1( x_i),\quad m\le  n. 
\end{split}
\end{equation}
So, stopping time $\tau$ is defined as follows:
\begin{itemize}
\item
if  ${\rm H}_0^{n}$ is accepted, the observations are continued,
 i.e., we test   ${\rm H}_0^{n+1}$
 vs. ${\rm H}_1^{n+1}$;\item If  ${\rm H}_1^{n}$ is accepted, then
we stop and   $\tau=n$.
\end{itemize}

\end{itemize}

In order to motivate  our idea of  robust stopping times, 
we discuss very briefly  basic statistical properties of the  above mentioned    approaches. 
 
\subsection{A Bayes approach}
Usually in this  approach the geometric a priori   distribution  
\[
\pi_m= \gamma(1-\gamma)^{m-1},\quad m=1,2,\ldots,\quad \gamma>0,
\]
 is used. Positive parameter $\gamma$ is assumed to  be known.
In this case, the optimal stopping time is given by the following famous theorem  \cite{Shir}: 
\begin{thm} \label{Th1} The optimal Bayes stopping time (see \eqref{eq2}) is given by
\begin{equation} \label{eq4}
\tau_\gamma^\alpha = \min\{k:\bar{\pi}(X^k)\ge 1- \alpha_\gamma\},
\end{equation}
where
\[
\bar{\pi}_\gamma(X^{k})= \mathbf{P}\bigl\{\theta \le k|X^{k}\bigr\},
\]
and $\alpha_\gamma\approx \alpha$ is a constant.
\end{thm}

Notice  that the geometric a priori distribution results in the  following recursive  formula for  a posteriori probability (see, e.g., \cite{Shir}): 
\begin{equation}\label{eq5}
\begin{split}
&\bar{\pi}_\gamma(X^k)= \\ &\quad=\frac{[\gamma+(1-\gamma) \bar{\pi}_\gamma(X^{k-1}) ]p_1(X_k) }{[\gamma+(1-\gamma) \bar{\pi}_\gamma(X^{k-1}) ]p_1(X_k) +[1-\bar{\pi}_\gamma(X^{k-1})](1-\gamma)p_0(X_k)}.
\end{split}
\end{equation}
So, if we denote for brevity
\[
\rho_\gamma(X^k)=\frac{\bar{\pi}_\gamma(X^k)}{1-\bar{\pi}_\gamma(X^k)},
\]
then  \eqref{eq5} may be rewritten in the following equivalent form:
\begin{equation}\label{eq6}
\rho_\gamma(X^k)=\frac{\gamma+\rho_\gamma(X^{k-1})}{1-\gamma}\times\frac{p_1(X_k)}{p_0(X_k)}.
\end{equation}

From this equation   we see, in particular, that the Bayes stopping time  depends  on $\gamma$ that  is hardly known in practice. 
In   statistics, in order to avoid such dependence,  the uniform
a priori distribution is usually used. Let's look how  this idea works in change point detection.    
The uniform a priori distribution assumes that $\gamma=0$ and  in this case we obtain immediately from \eqref{eq6} 
\begin{equation*}
\begin{split}
\rho_0(X^k) =\rho_0(X^{k-1})\times \frac{p_1(X_k)}{p_0(X_k)}.
\end{split}
\end{equation*} 
Therefore, for
\[
L_0(X^k)=\log[\rho_0(X^k)],
\]
we get 
\[
L_0(X^k)=\sum_{i=1}^k\log \frac{p_1(X_i)}{p_0(X_i)}.
\]

Hence, the optimal stopping time in the case of the uniform a priori distribution is given by
\begin{equation}\label{eq7}
\tau^\alpha_\circ =\min \bigl\{k: L_0(X^k)\ge t^\alpha\bigr\},
\end{equation}
where $t^\alpha$ is some constant. Fig.  \ref{Fig.2} shows a typical  trajectory of  $L_0(X^k), \, k=1,2,\ldots $, in detecting change in the Gaussian distribution with $\theta=80$. 

Computing the false alarm probability for this stopping time is not difficult and  based on the following simple fact.  Let
\[
\phi(\lambda)=\mathbf{E}_\infty\exp\biggl[\lambda \log\frac{p_1(X_1)}{p_0(X_1)}\biggr].
\]
\begin{lemma} \label{l1}For any $\lambda>0$
\begin{equation*}
\mathbf{E}_\infty\exp\bigl\{-\tau_\circ^\alpha\log[\phi(\lambda)]\bigr\}
\mathbf{1}\bigl(\tau_\circ^\alpha<\infty\bigr)\le 
\exp(-\lambda t^\alpha).
\end{equation*}
\end{lemma}

It follows immediately from the definition of $\phi(\lambda)$ that if $\lambda=1$, then $\phi(\lambda)=1$. So, by this Lemma we get
\[
\mathbf{P}_\infty\bigl\{\tau_\circ^\alpha<\infty\bigr\}\le \exp(-t^\alpha).
\]

As to the average  detection delay, it can  be easily computed with the help of the famous Wald
identity \cite{W,B}.   
The next theorem summarizes principal properties of  $\tau^\alpha_\circ$.
Let  us assume that
\[
\mu_0\stackrel{\rm def}{=}\int  \log \frac{p_0(x)}{p_1(x)} p_0(x)\,dx >0
\quad\text{and}\quad 
\mu_1\stackrel{\rm def}{=}\int  \log \frac{p_1(x)}{p_0(x)} p_1(x)\,dx >0.
\]

\begin{thm} Let 
$
t^\alpha=\log(1/\alpha).
$
Then for  $\tau_\circ^\alpha$ defined by \eqref{eq7} we have
\begin{align*}
&\alpha(\theta;\tau_\circ^\alpha)\le \alpha,
\\
&\Delta(\theta;\tau_\circ^\alpha)= \frac{\log(1/\alpha)+\theta\mu_0}{\mu_1}.
\end{align*}
\end{thm}

Fig. \ref{Fig.2} illustrates this theorem showing  $L_0(X^k), \, k=1,\ldots,200$, in the case 
of the change  in the mean of  the Gaussian distribution with $\theta=80$.
\begin{figure}
\centering
\includegraphics[angle=0,width=0.9\textwidth,height=0.35\textheight]{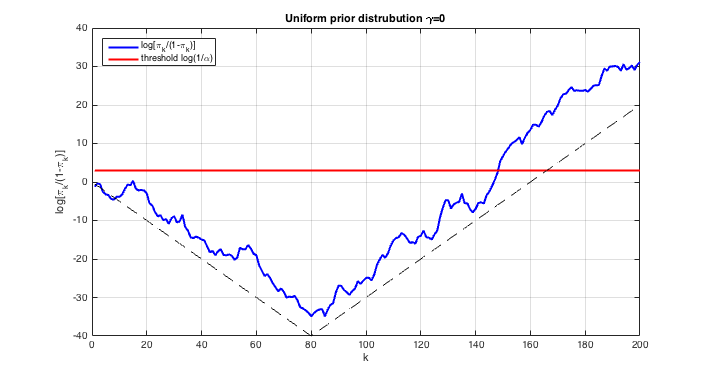}
\caption{Detecting change in the mean of Gaussian distribution with the help of $\tau^\alpha_\circ$. \label{Fig.2} }
\end{figure}

We would like to
emphasize that the fact that $\Delta(\theta;\tau_\circ^\alpha)$ is linear in $\theta$ is  not good from  practical and theoretical viewpoints. 
In order to understand why it is so,
let us now turn back to the Bayes setting  assuming that $\gamma>0$. In this case the following theorem holds true.
\begin{thm} \label{Th3}  Suppose $\gamma>0$.
Then for  $\tau^\alpha_\gamma$ defined by  \eqref{eq4} we have
\begin{align}
&\max_{\theta\in \mathbb{Z}^+}\alpha(\theta;\tau_\gamma^\alpha)=1, \nonumber
\\
&\Delta(\theta;\tau_\gamma^\alpha) =  \frac{\log[1/(\gamma \alpha)]}{\mu_1}+O(1),\quad 
\text{as} \quad \gamma, \alpha\rightarrow 0.
\label{eq8}
\end{align}
\end{thm}

This theorem may be  proved with the help of the standard techniques  described, e.g., in 
\cite{BN}. 

Fig. \ref{Fig.3} illustrates typical  behavior of $\log[\rho_\gamma(X^k)] $ with $\gamma>0$.
Notice that if $\tau_\circ^\alpha$ is used  in the considered case, then we obtain by \eqref{eq8} 
\[
\mathbf{E}\Delta(\theta;\tau_\circ^\alpha)= \frac{\log(1/\alpha)}{\mu_1}+\frac{\mu_0}{\mu_1}\times \frac{1}{\gamma}.
\]
So, we see that this mean detection delay  is far away from the optimal Bayes one given by
\[
\mathbf{E}\Delta(\theta;\tau_\gamma^\alpha) =  \frac{\log(1/\alpha)}{\mu_1}+\frac{1}{\mu_1}\times\log{\frac{1}{\gamma}}+O(1),\quad 
\text{as} \quad \gamma, \alpha\rightarrow 0.
\]

\begin{figure}
\centering
\includegraphics[angle=0,width=0.9\textwidth,height=0.35\textheight]{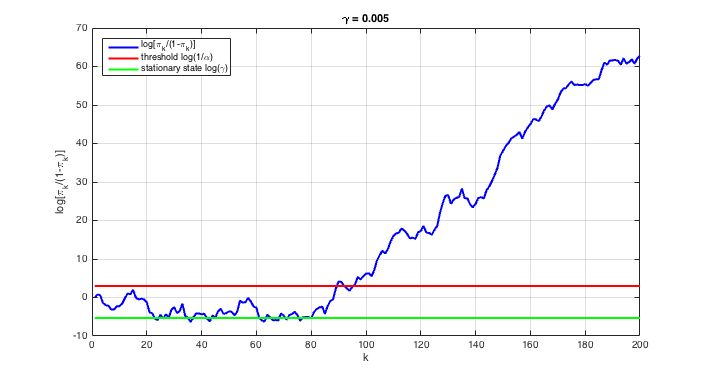}
\caption{Detecting change in the mean of Gaussian distribution with the help of $\tau^\alpha_\gamma$ ($\gamma=0.005$). \label{Fig.3} }
\end{figure}

Let us now summarize briefly main facts related  to the classical Bayes approach.
\begin{itemize}
\item if $\gamma=0$, then the average detection delay of the Bayes stopping time grows linearly in $\theta$;   
\item when $\gamma>0$,  the maximal false alarm probability  is not controlled.
\end{itemize}
In view of these facts it is clear that the   standard Bayes technique cannot provide  reasonable solutions to \eqref{eq1}.  
\subsection{A hypothesis testing approach}
The idea of this approach    is  based on the well-known sequential testing of two simple hypothesis \cite{W1}. 
However, we would like to emphasize that in contrast to the standard setting in \cite{W1}, in the change-point detection, this approach  has  a rather  heuristic character  since  here we  test a simple hypothesis   versus a compound alternative whose complexity grows with the  observations volume.   

In sequential hypothesis testing there are two  common  methods
\begin{itemize}
\item maximum likelihood;
\item Bayesian.
\end{itemize}

The maximum likelihood  test accepts   hypothesis ${\rm H}_1^n$ (see \eqref{eq3})    when
\begin{equation*}
\max_{k\le n}\frac{\prod_{i=1}^{k-1}  p_0(X_i)
\prod_{i=k}^n  p_1(X_i)} {\prod_{i=1}^n p_0(X_i)} \ge t^\alpha
\end{equation*}
or, equivalently,
\begin{equation*}
M(X^{n}) \ge t^\alpha,
\end{equation*}
 where 
\[
M(X^{n})=\max_{k \le n}
\sum_{i=k}^{n} \log\frac{ p_1(X_i)} { p_0(X_i)}.
\]
The threshold  $t^\alpha$ is  computed as  follows
\begin{equation*} 
t^\alpha=\min\Bigl\{t:\mathbf{P}_\infty\bigl\{
M(X^{n}) \ge t \bigr\} \le  \alpha\Bigr\},
\end{equation*}
where $\alpha$ is the first type error probability.
Notice that by Lemma \ref{l1} 
\[
\mathbf{P}_\infty\bigl\{
M(X^{n}) \ge x \bigr\}\le \exp(-x).
\]
Therefore the maximum likelihood test results in the following 
stopping time: 
\begin{equation}\label{eq9}
\tau_{\rm ml}^\alpha=\min\biggl\{n:\, 
M(X^{n}) \ge  \log \frac{1}{\alpha} \biggr\} .
\end{equation}

Notice also that   $M(X^{n})$ admits a simple recursive computation  \cite{Page}.  
Indeed, notice 
\begin{equation*}
\begin{split}
& \max_{k \le n}
\sum_{i=k}^{n} \log\frac{ p_1(X_i)} { p_0(X_i)}
\\& \quad =\max\biggl\{\log\frac{ p_1(X_{n})} { p_0(X_{n})},\log\frac{ p_1(X_{n})} { p_0(X_{n})}+\max_{k \le n-1}
\sum_{i=k}^{n-1} \log\frac{ p_1(X_i)} { p_0(X_i)}\biggr\}\\
&\quad = \log\frac{ p_1(X_{n})} { p_0(X_{n})}+\max\biggl\{0,\max_{k \le n-1}
\sum_{i=k}^{n-1} \log\frac{ p_1(X_i)} { p_0(X_i)}\biggr\}.
\end{split}
\end{equation*}
Therefore
\begin{equation} \label{eq10}
M(X^{n}) =\log\frac{ p_1(X_{n})} { p_0(X_{n})}+\bigl[
M(X^{n-1})\bigr]_+.
\end{equation}

This  method is usually called CUSUM algorithm.
It is well known that it is optimal in Lorden \cite{L} sense, i.e.,  for properly chosen $\alpha$, $\tau_{\rm ml}^\alpha$ minimizes 
\[
\sup_{\theta \in \mathbb{Z}^+}{\rm ess} \sup \mathbf{E}_\theta \bigl[(\tau-\theta)_+|X_1,\ldots,X_{\theta-1}\bigr]
\] 
in the class of stopping times $\bigl\{\tau: \mathbf{E}_\infty \tau\ge T \bigr\}$, see \cite{M}.

However, with this method cannot control the false alarm probability as shows the following theorem.
\begin{thm}\label{th-max}
For any $\alpha\in (0,1)$
\begin{align}
&\max_{\theta\in \mathbb{Z}^+}\alpha(\theta;\tau_{\rm ml}^\alpha)=1. \nonumber
\end{align}
As $\alpha\rightarrow 0$
\begin{align}
&\Delta(\theta;\tau_{\rm ml}^\alpha)= \frac{1+o(1)}{\mu_1}\log\frac{1}{\alpha}.
\nonumber
\end{align}
\end{thm}

The  Bayesian test is based on the assumption that $\theta$ is uniformly distributed on $[1,n]$. So, this test  accepts ${\rm H}_1^n$ 
when
\begin{equation}\label{eq11}
S(X^{n})\stackrel{\rm def}{=}\sum_{k=1}^{n}\frac{\prod_{i=1}^{k-1}  p_0(X_i)
\prod_{i=k}^n  p_1(X_i)} {\prod_{i=1}^n p_0(X_i)} \ge  t^\alpha.
\end{equation}
Since
\begin{equation*}
S(X^{n})=
\sum_{k=1}^{n}
\prod_{i=k}^n \frac{ p_1(X_i)} {p_0(X_i)},
\end{equation*}
and
\begin{equation*}
\begin{split}
\sum_{k=1}^{n}
\prod_{i=k}^n \frac{ p_1(X_i)} {p_0(X_i)}&=\sum_{k=1}^{n-1}
\prod_{i=k}^n \frac{ p_1(X_i)} {p_0(X_i)}+\frac{ p_1(X_{n})} {p_0(X_{n})}\\
&
 =\biggl[1+ \sum_{k=1}^{n-1}
\prod_{i=k}^{n-1} \frac{ p_1(X_i)} {p_0(X_i)} \biggr]
\frac{ p_1(X_n)} {p_0(X_n)},
\end{split}
\end{equation*}
   the test statistics in \eqref{eq11} 
admits the following recursive  computation:
\begin{equation*}
S(X^{n})=\bigl[1+ 
S(X^{n-1}) \bigr]\times
\frac{ p_1(X_n)} {p_0(X_n)}.
\end{equation*}

So, the corresponding stopping time is given by
\begin{equation*}
\tau_{\rm S}^\alpha =\min\bigl\{k: S(X^k)\ge  t^\alpha \bigr\}.
\end{equation*}
 
In the literature, this method is known as Shirayev-Roberts (SR) algorithm. It was firstly proposed in \cite{Sh1961} and \cite{R1966}.
 In \cite{PT2009} and \cite{FS} it was shown that it minimizes the integral average delay 
\[
\frac{1}{\mathbf{E}_\infty \tau}\sum_{\theta=1}^\infty \mathbf{E}_\theta(\tau-\theta)_+
\]
over all stopping times $\tau$ with
$
\mathbf{E}_\infty\tau\ge T.
$
 More detailed statistical properties  of SR procedure can be found in \cite{PT}.

As one can see on Fig. \ref{Fig.4},  in practice, there is no significant 
difference between CUSUM and 
SR algorithms.  
\begin{figure}
\centering
\includegraphics[angle=0,width=0.8\textwidth,height=0.35\textheight]{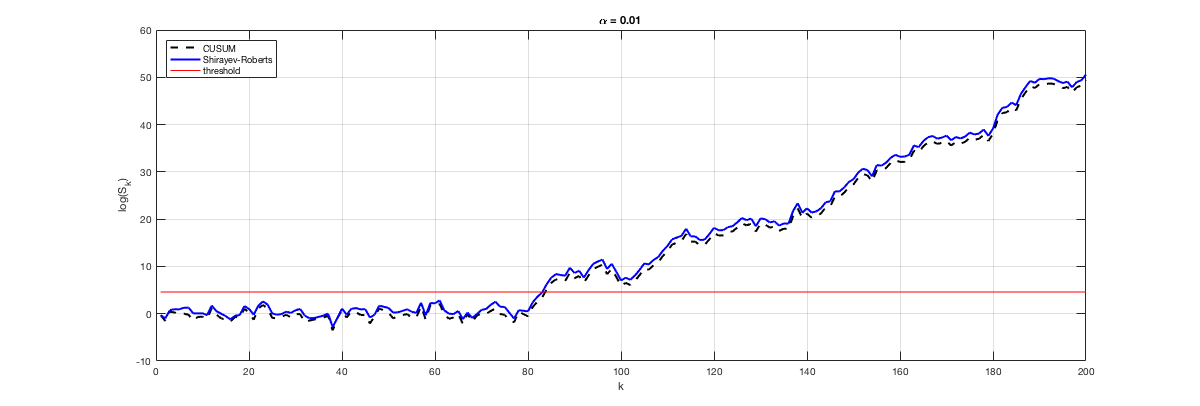}
\caption{Detecting change in the mean of Gaussian distribution with the help of CUSUM and SR procedures. \label{Fig.4} }
\end{figure}
  
  Notice also that for SR method the fact similar to Theorem 
\ref{th-max} holds true.
So,  the standard hypothesis testing methods results in stopping times with uncontrollable false alarm probabilities.

 \section{Robust stopping times}

The main idea in this paper is to make use  of multiple hypothesis testing methods for constructing stopping times.  
This can be done very easily by  replacing  the constant threshold  in the ML test \eqref{eq9} by one depending on $k$.  So, we define the  stopping time  
\begin{equation*}\label{rob-stopping}
\widetilde{\tau}^\alpha=\min\bigl\{k:\, 
M(X^{k}) \ge t^\alpha(k) \bigr\} .
\end{equation*}

In order to control the false alarm probability  and to obtain a nearly minimal average detection delay,  we are looking for    a \textit{minimal function} $t^\alpha(k), \ k=1,2,\ldots$, such that
\begin{equation*}
\mathbf{P}_\infty\bigl\{ \max_{k\ge \mathbb{Z}^+} \bigl[M(X^k) - t^\alpha(k)\bigr]
\ge 0 \bigr\} \le \alpha.
\end{equation*}

We begin our  construction of  $t^\alpha(k)$ with the following 
function:
\[
\varphi(x)=1+\log(x), \quad x\in \mathbb{R}^+,
\]
and define $m$-iterated  $\varphi(\cdot)$ by
\[
\Phi_{m}(x)=\varphi\bigl[\Phi_{m-1}(x)\bigr],\ \text{with}\ \Phi_1(x)=\varphi(x).
\]
Next, for given  $\epsilon\in(0,1)$,
define
\begin{equation}\label{eq12}
b_{m,\epsilon}(x)=-\log\bigg[
\frac{1}{\epsilon\Phi^\epsilon_m(x)}-\frac{1}{\epsilon\Phi^\epsilon_m(x+1)}\biggr],\quad x\in \mathbb{R}^+.
\end{equation}

Consider the following random variable:
\[
\zeta_{m,\epsilon} =\max_{k\in \mathbb{Z}^+} \bigl\{M(X^k)-b_{m,\epsilon}(k)\bigr\}.
\]
The next theorem plays a cornerstone role in our construction of robust stopping times.
\begin{thm} \label{th.5}
For any $\epsilon\in (0,1)$,  $m\ge 1$, 
and  $x> -\log(1-0.2075/2)\approx 0.11$ 
\begin{equation*}
\mathbf{P}\bigl\{\zeta_{m,\epsilon}\ge x\bigr\}\le 1-\exp\bigl\{-{\rm e}^{-x}\bigl[\epsilon^{-1}+{\rm e}^{-x}\bigr]\bigr\}.
\end{equation*}
\end{thm}

Therefore we can define  the quantile of order $\alpha$ of $\zeta_{m,\epsilon}$ by
\[
t^\alpha_{m,\epsilon} =\min\bigl\{x:\mathbf{P}\bigl\{\zeta_{m,\epsilon}\ge x \bigr\}\le \alpha\bigr\}.
\]
Fig. \ref{Fig.5} shows the distribution functions and quantiles of $\zeta_{1,\epsilon}$ for $\epsilon=\{0.01,0.2,1\}$ computed with the help of Monte-Carlo method. 

\begin{figure}
\centering
\includegraphics[angle=0,width=0.495\textwidth,height=0.35\textheight]{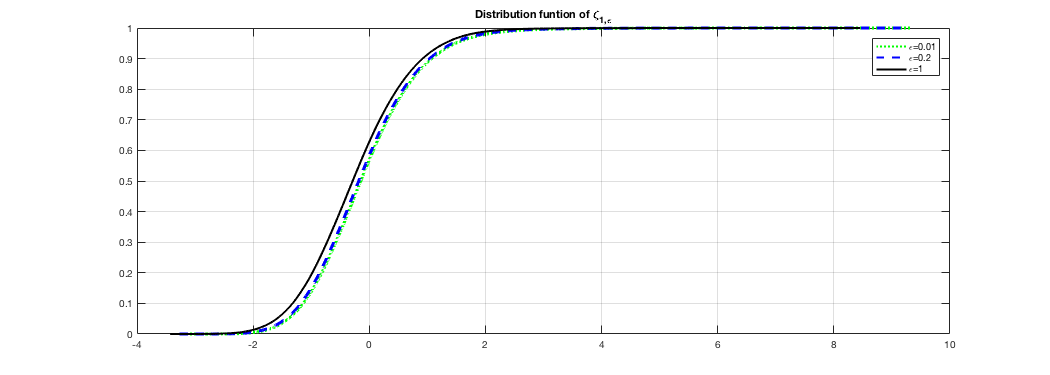}
\hfill
\includegraphics[angle=0,width=0.495\textwidth,height=0.35\textheight]{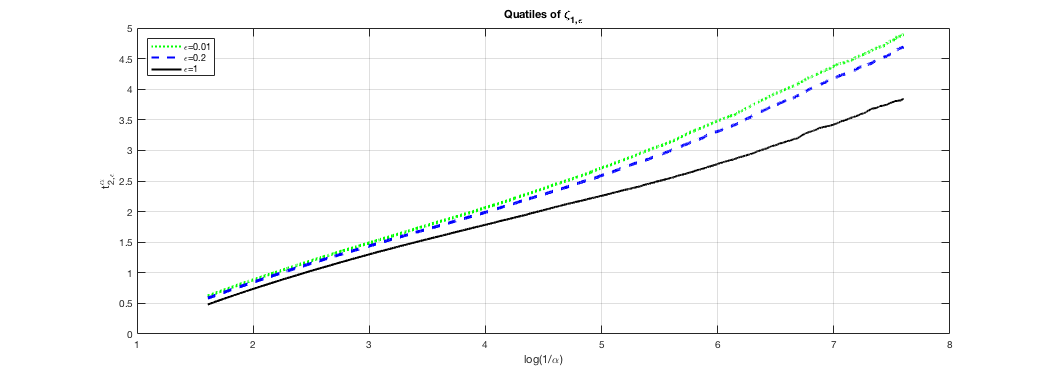}
\caption{Distribution functions 
 and quantiles of $\zeta_{1,\epsilon}$. \label{Fig.5} }
\end{figure}

The next theorem describes principal properties of the stopping time
\[ 
\widetilde{\tau}^\alpha_{m,\epsilon}=\min\bigl\{k: 
M(X^{k})\ge b_{m,\epsilon}(k)+t^\alpha_{m,\epsilon}\bigr\}.
\]

\begin{thm} \label{th-main} 
For any $\epsilon\in (0,1]$
\begin{equation*}
\begin{split}
&\alpha\bigl(\theta;\widetilde{\tau}^\alpha_{m,\epsilon}\bigr)\le \alpha,
\\
&\Delta\bigl(\theta;\widetilde{\tau}^\alpha_{m,\epsilon}\bigr)\le d^\alpha_{m,\epsilon}(\theta),  
\end{split}
\end{equation*}
where $d^\alpha_{m,\epsilon}(\theta)$ is a solution to 
\begin{equation}\label{eq13}
\mu_1 d^\alpha_{m,\epsilon}(\theta)
=b_{m,\epsilon}\bigl[\theta+d^\alpha_{m,\epsilon}(\theta)\bigr]+t^\alpha_{m,\epsilon}.
\end{equation}
\end{thm}

The asymptotic behavior of the average delay is described by  the following theorem
\begin{thm}\label{th.7}  For any $\epsilon\in (0,1]$, as ${\alpha}\rightarrow 0 $ and $\theta\rightarrow\infty$
\begin{equation}\label{eq14.0}
\Delta\bigl(\theta;\widetilde\tau_{m,\epsilon}^\alpha\bigr)\le \frac{1}{\mu_1}\biggl\{ \log \frac{\theta}{\alpha}+\sum_{j=1}^m
 \log[\Phi_j(\theta)] +\epsilon \log[\Phi_m(\theta)]
+\log\frac{1}{\epsilon}\biggr\} +o(1).
\end{equation}
\end{thm}

\medskip
\noindent
\textbf{Remark.} It is easy to check with a simple algebra that for any given $\theta>1$
\[
\lim_{j\rightarrow\infty} j \log[\Phi_j(\theta)] =2.
\]

\medskip
The robustness of $\widetilde{\tau}_{m,\epsilon}^\alpha$ w.r.t.  a priori geometric distribution of $\theta$ follows now almost immediately from  \eqref{eq14.0}. 
Indeed, suppose $\theta$ is a random variable with $$\mathbf{P}\bigl\{\theta=k\bigr\}=\gamma(1-\gamma)^{k-1}, \quad k\in \mathbb{Z}^+.$$ 
  Then, averaging  
\eqref{eq14.0}  w.r.t. this  distribution, we obtain 
\[
\mathbf{E} \Delta\bigl(\theta;\widetilde\tau_{m,\epsilon}^\alpha\bigr)\le  
\frac{1}{\mu_1}\biggl\{ \log \frac{1}{\alpha\gamma}+\sum_{j=1}^m \log\biggl[\Phi_j\biggl(\frac1\gamma\biggr) \biggr]
+\epsilon\log\biggl[ \Phi_m\biggl(\frac1\gamma\biggr)\biggr]+\log\frac{1}{\epsilon}\biggr\} +o(1)
\] 
as $\alpha,\gamma\rightarrow 0$,
and with \eqref{eq8} we arrive at 
\begin{thm}
As $\alpha,\gamma\rightarrow 0$
\begin{equation*}
\begin{split}
\mathbf{E} \Delta\bigl(\theta;\widetilde\tau_{m,\epsilon}^\alpha\bigr)\le&\mathbf{E} \Delta\bigl(\theta;\tau_{\gamma}^\alpha\bigr)+
\frac{1}{\mu_1}\biggl\{\sum_{j=1}^m \log \biggl[\Phi_j\biggl(\frac1\gamma\biggr) \biggr]+\epsilon 
\log\biggl[\Phi_m\biggl(\frac1\gamma\biggr)\biggr]+\log\frac{1}{\epsilon}\biggr\} +O(1)\\
= &(1+o(1))\mathbf{E} \Delta\bigl(\theta;\tau_{\gamma}^\alpha\bigr),
\end{split}
\end{equation*}
where $\tau_{\gamma}^{\alpha}$ is the optimal Bayesian stopping time (see Theorem \ref{Th1}).
\end{thm}

\appendix

\section{Appendix section}\label{app}

\begin{proof}[Proof of Lemma \ref{l1}]
 Since
\[
Y_k=\exp\bigl\{-k\log[\phi(\lambda)]+\lambda L_0(X^k)\bigr\}
\] 
is a martingale with $\mathbf{E}_\infty Y_k=1$, we have
\begin{equation*}
\begin{split}
1=& \mathbf{E}_\infty Y_{\tau_\circ^\alpha}=  \mathbf{E}_\infty Y_{\tau_\circ^\alpha}\mathbf{1}(\tau_\circ^\alpha<\infty)+  \mathbf{E}_\infty Y_{\tau_\circ^\alpha}\mathbf{1}(\tau_\circ^\alpha=\infty)\\ \ge &
\mathbf{E}_\infty Y_{\tau_\circ^\alpha}\mathbf{1}(\tau_\circ^\alpha<\infty)= 
\mathbf{E}_\infty \exp\bigl\{-\tau_\circ^\alpha\log[\phi(\lambda)]+\lambda A\bigr\}\mathbf{1}(\tau_\circ^\alpha<\infty). 
\end{split}
\end{equation*}
\end{proof}
In what follows we denote by $e_k$ be i.i.d. standard  exponential random variables. 
\begin{lemma}  \label{Lem-A.1} For any $m\ge 1$ 
and  $x> -\log(1-0.2075/2)\approx 0.11$
\begin{equation*}
\begin{split}
\mathbf{P}\Bigl\{\max_{k\in \mathbb{Z}^+}[e_k-b_{m,\epsilon}(k)]\ge x\Bigr\}\le 
1-\exp\Bigl\{-{\rm e}^{-x}\bigl[\epsilon^{-1} +{\rm e}^{-x}\bigr]\Bigr\},
\end{split}
\end{equation*}
where $b_{m,\epsilon}(\cdot)$ is defined by \eqref{eq12}.
\end{lemma}
\begin{proof} It is easy to check with a simple algebra that for any $u\in[0,1)$
\[
\log(1-u)\ge -u-\frac{u^2}{2(1-u)}.
\]
Therefore with this inequality
 we obtain
\begin{equation}\label{eq14}
\begin{split}
&\mathbf{P}\bigl\{\max_{k\in \mathbb{Z}^+}[e_k-b_{m,\epsilon}(k)]\ge x\bigr\}=
1-\prod_{k=1}^\infty \Bigl\{1-\mathbf{P}\bigl\{e_k\ge x+b_{m,\epsilon}(k)\bigr\}\Bigr\}\\
&\quad =1-\exp\biggl\{\sum_{k=1}^\infty \log\Bigl[1-{\rm e}^{-x-b_{m,\epsilon}(k)}\Bigr]\biggr\}\\&\quad \le 
1-\exp\biggl\{-{\rm e}^{-x}\sum_{k=1}^\infty {\rm e}^{-b_{m,\epsilon}(k)}-\frac{{\rm e}^{-2x}}{2(1-{\rm e}^{-x})}\sum_{k=1}^\infty {\rm e}^{-2b_{m,\epsilon}(k)}\biggr\}.
\end{split}
\end{equation}
It follows immediately from the definition of $b_{m,\epsilon}$, see \eqref{eq12}, that
\[
\sum_{k=1}^\infty {\rm e}^{-b_{m,\epsilon}(k)}=\frac{1}{\epsilon \Phi_m(1)}=\frac{1}{\epsilon}.
\]
It is also easy to check numerically that for any $m\ge 1$ and $\epsilon>0$
\[
\sum_{k=1}^\infty {\rm e}^{-2b_{m,\epsilon}(k)}< 0.2075.
\]
Therefore, substituting the above equations in \eqref{eq14}, we complete the proof.
\end{proof}

\begin{lemma} \label{Lem-A.2} For any $x>0$
\begin{equation*}
\mathbf{P}_\infty\Bigl\{\max_{k\in \mathbb{Z}^+} [M(X^k)-b_{m,\epsilon}(k)]\ge x\Bigr\}\le 
\mathbf{P}\Bigl\{\max_{k\in \mathbb{Z}^+}[e_k-b_{m,\epsilon}(k)]\ge x\Bigr\},
\end{equation*}
where random process $M(X^k)$ is defined by \eqref{eq10}.
\end{lemma}
\begin{proof}
Define   random integers $\kappa_1<\kappa_2<\ldots$ by
\[
\kappa_{k}=\min\bigl\{s>\kappa_{k-1}:\, M(X^{s})\le 0\bigr\},\quad t_0=0,
\]
From \eqref{eq10}  it is clear that these random variables are renovation points for the random process $ M(X^k)$ and therefore random variables
\[
\mu_k=\max_{\kappa_k< s\le \kappa_{k+1}} M(X^{s}),\quad k=1,2,\ldots.
\]
are independent. 
Since $b_{m,\epsilon}(k)$ is non-decreasing in $k$ and obviously $\kappa_k\ge k$, we get 
\begin{equation*}
\begin{split}
\max_{k\in \mathbb{Z}^+} [M(X^k)-b_{m,\epsilon}(k)]\le &\max_{k\in\mathbb{Z}^+} \max_{{\kappa_k< s\le \kappa_{k+1}}} [M(X^s)-b_{m,\epsilon}(t_k)]\\
 \le& \max_{k\in \mathbb{Z}^+}[\mu_k-b_{m,\epsilon}(k)].
\end{split}
\end{equation*}
Therefore, to finish the proof, it suffices to notice that by  \eqref{eq10} and Lemma \ref{l1}
$$
\mathbf{P}_\infty\bigl\{\mu_k\ge x\bigr\}\le \mathbf{P}_\infty\biggl\{ \max_{k\in \mathbb{Z}^+}\sum_{s=\theta}^k \log \frac{p_0(X_s)}{p_1(X_s)}\ge x \biggr\}\le \exp(-x).
$$
\end{proof}

Theorem 
 \ref{th.5} follows now immediately from Lemmas \ref{Lem-A.1}, \ref{Lem-A.2}. 
\begin{proof}[Proof of Theorem \ref{th-main}]
It follows  from 
\eqref{eq10} that for all $k\ge \theta$
\begin{equation*}
M(X^k)\ge \sum_{s=\theta}^k \log \frac{p_0(X_s)}{p_1(X_s)}
\end{equation*}
and therefore 
\[
\Delta(\theta;\widetilde{\tau}_{m,\epsilon})\le \mathbf{E}_\theta \tau^+,
\]
where 
\[
\tau^+ =\min\biggl\{k\ge 1:\sum_{s=\theta}^{\theta+k} \log \frac{p_0(X_s)}{p_1(X_s)} \ge b_{m,\epsilon}(\theta+k)+t^\alpha_{m,\epsilon} \biggr\}.
\]

Computing  $\mathbf{E}_\theta \tau^+$ is based on  the famous Wald's identity \cite{W} (see also \cite{B}). For given $\theta\in \mathbb{Z}^+, m\in \mathbb{Z}^+, \epsilon>0$, define  function 
\[
B(k)=b_{m,\epsilon}(\theta+k)+t^\alpha_{m,\epsilon}, \ k\in \mathbb{Z}^+.
\] 
It is clear that $B(\cdot)$ is a convex function and therefore for any $k_0\ge 1$
\[
B(k)\le B(k_0)+B'(x_0)(k-k_0).
\]  
Hence,
\[
\tau^+ \le\tau^{++} =\min\biggl\{k\ge 1:\sum_{s=\theta}^{\theta+k} \log \frac{p_0(X_s)}{p_1(X_s)} \ge B(k_0)+B'(k_0)(k-k_0) \biggr\}.
\]

Next, we obtain by Wald's identity
\[
\mu_1 \mathbf{E}_\theta\tau^{++} \le  B(k_0)+ B'(k_0)\bigl( \mathbf{E}_\theta\tau^{++}-k_0\bigr)
\]
and thus
\begin{equation}\label{eq15}
\mathbf{E}_\theta \tau^{++}\le \frac{B(k_0)-B'(x_0)k_0}{\mu_1-B'(k_0)}.
\end{equation}

To finish the proof, we choose  $k_0=d_{m,\epsilon}^\alpha(\theta)$ (see \eqref{eq13}), and 
notice that
$ B(k_0)=\mu_1 k_0$. Hence, by \eqref{eq15} 
\[
\mathbf{E}_\theta\tau^{++}\le k_0=d_{m,\epsilon}^\alpha(\theta).
\]
\end{proof}

\begin{proof}[Proof of Theorem \ref{th.7}]

It follows immediately from Theorem \ref{th.5} that as  $\alpha \rightarrow 0$
\begin{equation}\label{eq16}
t_{m,\epsilon}^\alpha \le  \log\frac{1}{\alpha\epsilon}+o(1). 
\end{equation}

Next, by convexity of $b_{m,\epsilon}(\cdot)$ 
we obtain for   any $x,x_0$ 
\[
b_{m,\epsilon}(\theta+x)\le b_{m,\epsilon}(\theta+x_0)+b_{m,\epsilon}'(\theta+x_0)(x-x_0).
\]
Therefore, choosing 
\[
x_0=\frac{b_{m,\epsilon}(\theta)+t_{m,\epsilon}^\alpha}{\mu_1}
\]
 we get by \eqref{eq13}
\begin{equation}\label{eq17}
d_{m,\epsilon}^\alpha(\theta)\le \frac{b_{m,\epsilon}(\theta+x_0)+t_{m,\epsilon}^\alpha}{\mu_1-b'_{m,\epsilon}(\theta+x_0)}.
\end{equation}

So, our next step is to upper bound $b_{m,\epsilon}(\cdot)$.
First, notice that
\[
-\frac{1}{\epsilon}\frac{d \Phi^{-\epsilon}_m(x)}{dx}=\Phi^{-1-\epsilon}_m(x) \Phi_{m}'(x)=\frac{\Phi^{-\epsilon}_m(x)}{x}
 \prod_{j=1}^{m}\frac{1}{\Phi_j(x)},
\]
and thus
\begin{equation*}
-\log\biggl[ -\frac{1}{\epsilon}\frac{d \Phi^{-\epsilon}_m(x)}{dx} \biggr]=
\log(x)+\sum_{j=1}^m \log[\Phi_j(x)]+\epsilon\log[\Phi_m(x)].
\end{equation*}

Therefore it follows immediately from this equation and \eqref{eq12} that as $k \rightarrow\infty$ 
\begin{equation}\label{eq18}
b_{m,\epsilon}(k)= \log(k)+\sum_{j=1}^m \log[\Phi_j(k)]+\epsilon\log[\Phi_m(k)]+o(1).
\end{equation}
It is also easy to check that 
\begin{equation}\label{eq19}
b_{m,\epsilon}'(k)=O\biggl(\frac1k\biggr).
\end{equation}

Finally, substituting \eqref{eq16}, \eqref{eq18}, and \eqref{eq19} in \eqref{eq17}, we  complete the proof. 
\end{proof}


\end{document}